\documentclass{amsart}

\usepackage[margin=3.5cm]{geometry}

\usepackage{mathrsfs}
\usepackage{amsfonts}
\usepackage{amsmath}
\usepackage{amssymb}
\usepackage{mathtools}
\usepackage{float}
\usepackage{xcolor}
\usepackage[inline]{enumitem}
\usepackage{csquotes}
\usepackage{subcaption}
\usepackage{tikz}
\usepackage{tikz-cd}
\usetikzlibrary{arrows,automata}
\usepackage{hyperref}
\usepackage{centernot}
\usepackage{varwidth}
\usepackage{bbm}

\newcommand{\pres}[3]{\textnormal{#1} \langle #2 \mid #3 \rangle}

\newcommand{\xra}[1]{\xrightarrow{}^\ast_{#1}}

\newcommand{\sR}{\mathscr{R}}
\newcommand{\sS}{\mathscr{S}}

\DeclareFontFamily{U}{wncy}{}
    \DeclareFontShape{U}{wncy}{m}{n}{<->wncyr10}{}
    \DeclareSymbolFont{mcy}{U}{wncy}{m}{n}
    \DeclareMathSymbol{\sh}{\mathord}{mcy}{"78} 

\newcommand{\comment}[1]{\textcolor{red}{#1}}

\newtheorem{theorem}{Theorem} 

\newtheorem{lemma}{Lemma}     

\theoremstyle{definition}

\newtheorem{remark}{Remark}

\begin{document}

\title[A special monoid with undecidable Diophantine problem]{A word-hyperbolic special monoid with undecidable Diophantine problem}

\author{Carl-Fredrik Nyberg-Brodda}
\address{Alan Turing Building, Department of Mathematics, University of Manchester, UK.}
\email{carl-fredrik.nybergbrodda@manchester.ac.uk}

\thanks{The author gratefully acknowledges funding from the Dame Kathleen Ollerenshaw Trust, which is funding his current position as Research Associate at the University of Manchester.}


\date{\today}



\begin{abstract}
The Diophantine problem for a monoid $M$ is the decision problem to decide whether any given system of equations has a solution in $M$. In this note, we give a simple example of a context-free, word-hyperbolic, finitely presented, special monoid $M$ with trivial group of units, and such that the Diophantine problem is undecidable in $M$. This answers two questions asked by Garreta \& Gray in 2019, and shows that the decidability of the Diophantine problem in hyperbolic groups, as proved by Dahmani \& Guirardel, does not generalise to word-hyperbolic monoids. 
\end{abstract}

\vspace*{-0.2cm}
\maketitle

\noindent A monoid is said to be \textit{special} if all its defining relations are of the form $w=1$. The study of special monoids form a classical part of combinatorial semigroup theory, seeing intense study by Adian \cite{Adian1966} and his student Makanin \cite{Makanin1966}. The latter would subsequently study the \textit{Diophantine problem} in free monoids and groups. For a given monoid $M$, the Diophantine problem asks to decide whether a given system of equations has a solution in $M$. Makanin famously solved this problem  for free monoids \cite{Makanin1977} and free groups \cite{Makanin1982}. The Diophantine problem has seen a good deal of study since, particularly for groups; for example, Dahmani \& Guirardel \cite{Dahmani2010} recently proved the decidability of the problem for all hyperbolic groups. In this note, we shall prove that their result does not extend to word-hyperbolic monoids (in the sense of Duncan \& Gilman \cite{Duncan2004}). We shall prove that there exists a finitely presented special monoid $M$ with context-free word problem, and with trivial group of units, such that the Diophantine problem for $M$ is undecidable. This answers two questions asked by Garreta \& Gray in 2019 regarding the properties of the Diophantine problem in special monoids resp. word-hyperbolic monoids.  

We follow the notation of \cite{NybergBrodda2022c}, all of which is very standard. The free monoid on a finite set $A$ is denoted by $A^\ast$, and the free semigroup is denoted by $A^+$. We assume the reader is familiar with the basic elements of formal language theory (e.g. via \cite{Hopcroft1979}) and rewriting systems (e.g. via \cite{Book1993}). For a rewriting system $\sR \subseteq A^\ast \times A^\ast$, we let $M_\sR$ denote the monoid presented by the monoid presentation $\pres{Mon}{A}{\sR}$. If this presentation is special, then we also say that $\sR$ is special. Additionally, a special monoid $M$ generated by a finite set $A$ has \textit{context-free word problem} if the set of words over $A$ representing the identity element is a context-free language. This is, as proved (non-trivially) in \cite{NybergBrodda2022c}, equivalent to $M$ having context-free word problem in the sense of Duncan \& Gilman \cite{Duncan2004}. Finally, $M$ is \textit{word-hyperbolic} if there exists a regular language $R \subseteq A^\ast$ which generates $M$ and such that the ``multiplication table'' for $M$ with respect to $R$ is a context-free language (see \cite{Duncan2004} or e.g. \cite{NybergBrodda2022d}) for details). 

Our starting point is the following theorem:

\begin{theorem}[{Otto, \cite[Theorem~1]{Otto1995}}]\label{Thm:OttoTheorem}
There exists a finite special rewriting system $\sS$ such that $\sS$ has no non-trivial critical pairs, and the Diophantine problem is undecidable for $M_\sS$. 
\end{theorem}

We will show that any $M_\sS$ as in the statement of Theorem~\ref{Thm:OttoTheorem} has trivial group of units. Recall that the group of units of a special monoid $M$ is the subgroup consisting of all two-sided invertible elements of $M$, and is denoted by $U(M)$. The triviality of $U(M_\sS)$, when combined with a result by the author, will then yield Theorem~\ref{Thm:MainThm}. 

We first recall some elements of special monoid theory from Zhang \cite{Zhang1992} and \cite{NybergBrodda2022c}. Let $M = \pres{Mon}{A}{u_i = 1 \: (1 \leq i \leq k)}$ be a finitely presented special monoid, with all $u_i$ assumed non-empty. A word $w \in A^\ast$ representing a two-sided unit $m \in U(M)$ will be called \textit{invertible}, and analogously for \textit{left/right} invertible words. An invertible word $w \in A^+$ is called \textit{minimal} if none of its proper non-empty prefixes is invertible. The set of minimal words clearly forms a biprefix code generating the invertible words. Hence, every defining word $u_i$ can be uniquely factored over $A^\ast$ into minimal words as $u_i \equiv \lambda_{i,1} \lambda_{i,2} \cdots \lambda_{i,n}$. The collection $\bigcup_{i, j} \{ \lambda_{i,j} \}$ of all such factors arising in this way for $1 \leq i \leq k$ is denoted by $\Lambda$. Clearly, $\Lambda$ is a finite set of minimal words. The set $\Delta$ is defined as the set of minimal words $\delta \in A^+$ such that there exists some $\lambda \in \Lambda$ with $|\delta|\leq |\lambda|$ and $\delta = \lambda$ in $M$. Then $\Delta$ is a finite set, too, and $\Delta$ generates (as a submonoid) the subgroup $U(M)$ by \cite[Lemma~3.4]{Zhang1992}. Hence $\Lambda$ also generates $U(M)$. 

We now give a direct proof that $U(M_\sS)$ is trivial, with $M_\sS$ any monoid as in Theorem~\ref{Thm:OttoTheorem}. In fact, we give a more general result. 

\begin{lemma}\label{Lem:main_lemma}
Let $\sR$ be a finite special rewriting system. If $\sR$ has no non-trivial critical pairs, then $U(M_\sR) = 1$. 
\end{lemma}
\begin{proof}
Suppose that $\sR \subseteq A^\ast \times A^\ast$, and that the rules of $\sR$ are $(u_i, 1)$ with $u_i \in A^+$ for $1 \leq i \leq k$. As $\sR$ has no non-trivial critical pairs, it is locally confluent, and hence confluent by Newman's Lemma \cite{Newman1942}. Furthermore, as $\sR$ has no non-trivial critical pairs, every proper subword of $u_i$ is irreducible mod $\sR$ for $1 \leq i \leq k$. 

Fix some $1 \leq i \leq k$. Factorise $u_i \equiv \lambda_1 \lambda_2 \cdots \lambda_n$ uniquely into minimal factors $\lambda_j \in \Lambda$. Assume, for sake of contradiction, that $n > 1$. As $\lambda_n$ is invertible, and $\lambda_1 \lambda_2 \cdots \lambda_{n-1}$ is a left inverse of $\lambda_n$, it follows that it is also a right inverse of $\lambda_n$, i.e. $\lambda_n \lambda_1 \cdots \lambda_{n-1} = 1$ in $M_\sR$. As $\sR$ is complete, it follows that $\lambda_n \lambda_1 \cdots \lambda_{n-1} \xra{\sR} 1$. The words $\lambda_n$ and $\lambda_1 \cdots \lambda_{n-1}$ are irreducible, being proper subwords (as $n>1$) of $u_i$. Hence, to rewrite $\lambda_n \lambda_1 \cdots \lambda_{n-1}$ to the empty word $1$, there must be some $u$ with $(u, 1) \in \sR$ and $u$ a subword of $\lambda_n \lambda_1 \cdots \lambda_j$ for some $1 \leq j \leq n-1$, and such that this occurrence of $u$ overlaps non-trivially with both $\lambda_n$ and $\lambda_j$. As $\lambda_n$ is minimal, it has no non-trivial right invertible suffix; but every prefix of $u$ is right invertible, so the overlap of $u$ with $\lambda_n$ must be all of $\lambda_n$. Hence $u \equiv \lambda_n u'$ for some $u' \in A^+$. But then the rules $(u_i, 1), (u, 1) \in \sR$ overlap non-trivially, contradicting the lack of non-trivial critical pairs in $\sR$. Hence $n=1$. 

The above shows that every $u_i$ factorises into a single minimal invertible piece, i.e. $\Lambda  =\{ u_1, u_2, \dots, u_k \}$. As $\Lambda$ generates $U(M_\sR)$, and $u_i = 1$ in $M_\sR$ for every $i$, we have $U(M_\sR) = 1$. 
\end{proof}

Using Lemma~\ref{Lem:main_lemma}, we find the main result of this note:

\begin{theorem}\label{Thm:MainThm}
There exists a finitely presented special monoid such that: 
\begin{enumerate}
\item The group of units $U(M)$ is trivial;
\item $M$ has context-free word problem;
\item $M$ is word-hyperbolic; 
\item The Diophantine problem for $M$ is undecidable.
\end{enumerate}
\end{theorem}
\begin{proof}
Let $M = M_\sS$ for any system $\sS$ as in Theorem~\ref{Thm:OttoTheorem}. Then $U(M) = 1$ by Lemma~\ref{Lem:main_lemma}, and hence $M$ has context-free word problem by the main result of \cite{NybergBrodda2022c}. Alternatively, as $M$ is defined by a finite complete special rewriting system, the context-freeness of its word problem (in the sense of Duncan \& Gilman) follows from \cite[Corollary~3.8]{Book1982}. Finally, it is not difficult to show that any context-free monoid is word-hyperbolic, see e.g. the proof of \cite[Theorem~3]{Cain2012}.
\end{proof}

Let $M$ be a finitely presented special monoid with group of units $G$. In 2019, Garreta \& Gray asked two questions in the pre-print version of \cite{Garreta2021}\footnote{Available online at \texttt{arXiv:1908.00098v1}.}. The first (Question~6.4) asked: if the Diophantine problem is decidable in $G$, then does it follow that it is decidable in $M$? The second (Question~6.5) asked: is the Diophantine problem decidable for finitely presented word-hyperbolic monoids? As the trivial group clearly has decidable Diophantine problem, our Theorem~\ref{Thm:MainThm} hence gives a negative answer to both these questions, and in fact a negative answer to the more general version of Question~6.5 of whether the Diophantine problem is decidable in finitely presented context-free monoids.

\begin{remark}
In \cite{Otto1995}, Otto constructs a system $\sS$ as in Theorem~\ref{Thm:OttoTheorem} directly from a given undecidable instance of the Post correspondence problem. Using a small undecidable normal system due to Matiyasevich \cite{Matiyasevich1967} and a reduction due to Post \cite{Post1943}, one can find $\sS$ as in Theorem~\ref{Thm:OttoTheorem} with $71$ rules. That is, an explicit example of a $71$-relation special monoid $M$ satisfying Theorem~\ref{Thm:MainThm} can be constructed. It would be interesting to know whether an example $M = \pres{Mon}{A}{w=1}$ with a single relation can be constructed to satisfy the conclusions of Theorem~\ref{Thm:MainThm} (this is essentially Questions~6.2 and 6.3 in \cite{Garreta2021}). We suspect the answer to be negative. 
\end{remark}

{
\bibliography{Specialdiophantinetrivial.bib} 
\bibliographystyle{plain}
}
 \end{document}